\documentclass[12pt]{amsart}
\usepackage{amsmath,amssymb,amsthm}
\usepackage[margin=1.1in]{geometry}
\usepackage[colorlinks=true,linkcolor=blue,citecolor=blue]{hyperref}
\usepackage[capitalize]{cleveref}

\newcommand{\R}{\mathbb{R}}
\newcommand{\C}{\mathbb{C}}
\newcommand{\E}{\mathbb{E}}
\newcommand{\Lom}{\mathfrak{L}}
\newcommand{\sjz}{j_{\nu,1}}

\newtheorem{theorem}{Theorem}[section]
\newtheorem{proposition}[theorem]{Proposition}
\newtheorem{lemma}[theorem]{Lemma}
\newtheorem{corollary}[theorem]{Corollary}

\theoremstyle{remark}
\newtheorem{remark}[theorem]{Remark}

\title{Moment-optimal Gamma approximation of Bessel Process exit times}
\author{Steven P. Clark}
\address{University of North Carolina at Charlotte}
\email{spclark@charlotte.edu}

\subjclass[2020]{Primary 60J60, 60E10; Secondary 33C10, 33C45, 41A21, 41A55, 60E07}
\keywords{Bessel process, exit time, Gamma convolutions, Thorin measure, Lommel polynomials, Bessel functions,
Gaussian quadrature, Jacobi matrix, cumulant matching, Pad\'e approximation}

\begin{document}

\begin{abstract}
We construct finite Gamma convolutions approximating the first hitting time of $1$ by a Bessel process of index $\nu>-1$ started at the origin.
The scale measure in the classical exponential representation of this law is obtained by squaring the positive half of a dilated modified Lommel polynomial orthogonality measure.
Its Jacobi recurrence yields an $N$-factor Gamma convolution that uniquely matches the first $2N$ cumulants, the maximal number attainable in this class.
From an exact Gaussian quadrature remainder we derive supergeometric convergence of the Laplace transforms
on compact subsets of a common slit plane and geometric convergence of the distribution functions uniformly on the real line.
The approximants have total Gamma shape $N(N+\nu+1)$, and their Laplace transforms and negative moments decrease to those of the exit law.
\end{abstract}

\maketitle

\section{Introduction}\label{sec:intro}

Fix $\nu>-1$, and let $T$ be the first time a Bessel process of index $\nu$, started at the origin, reaches $1$.
For integer $d=2\nu+2$, this is the exit time of standard Brownian motion from the unit ball in $\R^d$.
Ciesielski and Taylor \cite{CiesielskiTaylor} for integer dimension, and more generally Kent \cite{Kent} for every index,
established an exponential representation of this law.
Writing $j_{\nu,1}<j_{\nu,2}<\cdots$ for the positive zeros of the Bessel function $J_\nu$ and setting
\begin{equation}\label{eq:ck}
c_k=\frac{2}{j_{\nu,k}^2},\qquad k\ge1,
\end{equation}
one has
\begin{equation}\label{eq:ct}
T\overset d=\sum_{k\ge1}c_ke_k,
\end{equation}
where the $e_k$ are independent exponential random variables of mean one.\footnote{
Here the Bessel process is $R_t=\sqrt{Y_t}$, where $Y$ is the nonnegative strong solution of
\begin{equation}\label{eq:besq}
dY_t=(2\nu+2)\,dt+2\sqrt{Y_t}\,dB_t,\qquad Y_0=0,
\end{equation}
and $B$ is a standard Brownian motion. This specifies the process for every $\nu>-1$; see \cite[Chapter XI]{RevuzYor}.
For $-1<\nu<0$ the origin is a regular boundary at which this solution is instantaneously reflecting. With absorption
at the origin instead, the zeros of $J_{-\nu}$ would replace those of $J_\nu$ in \cref{eq:ck} \cite[Table~2 and Eq. (8.2)]{Kent1980}.}
The representation is the Bessel case of Kent's eigenvalue expansion for diffusion hitting times
\cite[Corollary~5.1 and Section~8]{Kent1980}, in which the rates $1/c_k=j_{\nu,k}^2/2$ are the eigenvalues of the
generator on $(0,1)$ killed at $1$.
The corresponding Laplace transform is
\begin{equation}\label{eq:expectation_ct}
\phi(\lambda):=\E[e^{-\lambda T}] = \prod_{k\ge1}(1+\lambda c_k)^{-1} = \frac{1}{{}_0F_1(;\nu+1;\lambda/2)},
\end{equation}
where ${}_0F_1$ is the confluent hypergeometric limit function.
The identity holds for real $\lambda>-1/c_1$ and follows from the canonical product of $J_\nu$ \cite[\S15.41]{Watson}.

In this paper we seek finite Gamma convolutions that approximate the law of $T$.
Allowing both the shape and the scale of each factor to vary gives a class in which cumulant matching has an exact solution.
\Cref{thm:intro_match} characterizes that solution and establishes the maximal number of cumulants of $T$ it can match,
and \cref{thm:intro_convergence} measures the approximation error in the transforms and in the distribution functions.

The construction uses the finite positive scale measure
\[\rho_\nu :=\sum_{k\ge1}c_k\delta_{c_k},\quad\text{with}\quad \rho_\nu(\R)=\E[T]=\frac{1}{2\nu+2}.\]
For $N\ge1$, let $(\hat c_i,\hat w_i)_{i=1}^N$ be its $N$-point Gauss quadrature rule, with positive weights and nodes ordered as $0<\hat c_1<\cdots<\hat c_N<c_1$. Thus
\[\sum_{i=1}^N\hat w_i f(\hat c_i) = \int f(y)\,d\rho_\nu(y),\]
for every polynomial $f$ with $\deg f\le 2N-1 $.
Since this is the defining property of the Gauss rule, we shall henceforth refer to it simply as \emph{Gaussian exactness}.
We write $\Gamma(b,c)$ for the Gamma law with shape $b>0$ and scale $c>0$, whose Laplace transform is $(1+c\lambda)^{-b}$, and define
\[\widehat T_N = \sum_{i=1}^N G_i, \qquad G_i\sim\Gamma\!\left(\frac{\hat w_i}{\hat c_i},\hat c_i\right), \quad\text{independently}.\]
The Jacobi recurrence is derived in \cref{sec:scale}, and \cref{prop:quadrature} computes the nodes and weights from it.

We now state our main results.

\begin{theorem}
\label{thm:intro_match}
For every $\nu>-1$ and $N\ge1$,
\[\kappa_m(\widehat T_N)=\kappa_m(T),\qquad 1\le m\le2N,\]
where $\kappa_m$ denotes the $m$th cumulant. Among laws of sums of at most $N$ independent Gamma random variables with positive shapes and
scales, the law of $\widehat T_N$ is the unique law with this property.
The maximal number of consecutive cumulants, starting with the first, that a law in this class can match is $2N$.
\end{theorem}

Equivalently, the first $2N$ moments agree, which is the sense in which the approximation is moment-optimal.
For example, the one-point rule produces
\[\widehat T_1\sim\Gamma\!\left(\nu+2,\frac{1}{2(\nu+1)(\nu+2)}\right),\]
the Gamma law matching the mean and variance of $T$.

To state the convergence result, write $F$ and $F_N$ for the distribution functions of $T$ and $\widehat T_N$, respectively. Define
\[\phi_N(z)=\E[e^{-z\widehat T_N}]\quad\text{and}\quad \mathcal D_\nu=\C\setminus(-\infty,-1/c_1].\]
The transforms are initially defined on the right half-plane.
Let $s_*>0$ be the unique solution of $s_*+4\log s_*=6\log2-4$, and set $q_*=e^{-s_*}\approx 0.43049$.

\begin{theorem}
\label{thm:intro_convergence}
Fix $\nu>-1$.
\begin{enumerate}
\item[(i)] The transforms $\phi$ and $\phi_N$ admit analytic continuations to $\mathcal D_\nu$.
For every compact $K\subset\mathcal D_\nu$, there are finite constants $C_{\nu,K},A_{\nu,K}>0$ such that
\[\sup_{z\in K}|\phi_N(z)-\phi(z)| \le C_{\nu,K}\exp\!\left(-4N\log N+A_{\nu,K}N\right),\qquad N\ge1.\]
\item[(ii)] For every $q\in(q_*,1)$, there is a finite constant $C_{\nu,q}>0$ such that
\[\sup_{x\in\R}|F_N(x)-F(x)|\le C_{\nu,q}q^N, \qquad N\ge1.\]
\end{enumerate}
\end{theorem}

The construction rests on a correspondence between cumulants and moments.
Formally, if a Gamma convolution $X$ has shapes $b_i$ and scales $s_i$, then
\[\frac{\kappa_m(X)}{(m-1)!} = \sum_i b_i s_i^m = \int y^{m-1}\,d\mu_X(y),\quad\text{where}\quad \mu_X=\sum_i b_i s_i\delta_{s_i}.\]
For the exit law, the corresponding measure is $\rho_\nu$.
Matching the first $2N$ cumulants is therefore a positive $N$-point quadrature problem of degree $2N-1$, which is the problem solved by Gaussian quadrature.
The orthogonality of the node polynomial also accounts for uniqueness and maximality. These arguments will be used to establish \cref{thm:intro_match} in \cref{sec:approx}.

Finite Gamma approximations obtained by applying Pad\'e approximation to negative logarithmic derivatives of Laplace transforms were developed in
\cite{FurmanHackmannKuznetsov,FurmanKuznetsovMiles}. The construction here realizes that scheme for the Bessel exit law through a Lommel polynomial recurrence.
A special feature of this law is that the orthogonal polynomials of $\rho_\nu$ can be identified explicitly.
For $\beta>0$, let $\Lom_{\nu,\beta}$ be the symmetric measure placing mass $(\beta/j_{\nu,k})^2$ at each of the points $\pm\beta/j_{\nu,k}$.
Its monic orthogonal polynomials are dilated modified Lommel polynomials \cite{Dickinson}, with recurrence coefficients
\begin{equation}\label{eq:A2}
A_m^2=\frac{\beta^2}{4(\nu+m)(\nu+m+1)},\qquad m\ge1.
\end{equation}
At $\beta=\sqrt2$, the image of the positive half of this measure under $x\mapsto x^2$ is $\rho_\nu$.
The even reduction of the Lommel recurrence therefore supplies Jacobi coefficients that are rational in $\nu$ for the scale measure.
The associated finite Jacobi matrices determine the nodes and weights \cite{GolubWelsch}, so the Gamma parameters can
be computed directly from these coefficients, without evaluating Bessel zeros.

The same quadrature also controls the approximation error.
Let $Q$ and $Q_N$ be the analytic logarithms determined by $\phi=e^{-Q}$, $\phi_N=e^{-Q_N}$, and $Q(0)=Q_N(0)=0$.
On $\mathcal D_\nu$, their derivatives satisfy
\[Q'(z)=\int\frac{d\rho_\nu(y)}{1+zy}\quad\text{and}\quad Q_N'(z)=\sum_{i=1}^N\frac{\hat w_i}{1+z\hat c_i}.\]
By the Gaussian quadrature--Pad\'e correspondence \cite[Theorem~6]{HackmannKuznetsov},
$Q_N'$ is the $[N-1/N]$ Pad\'e approximant of $Q'$ at zero, and integration from zero recovers the Gamma product.
Since \cref{eq:expectation_ct} furnishes an exponential moment of $T$, expansion at zero is admissible in this scheme \cite[Section~3]{FurmanKuznetsovMiles}.
If $p_N$ is the monic orthogonal polynomial of degree $N$ for $\rho_\nu$, the Lommel recurrence determines the squared norm
\[h_N:=\int p_N(y)^2\,d\rho_\nu(y) = \frac{\Gamma(\nu+1)^2}{2^{2N+1}\Gamma(\nu+2N+1)\Gamma(\nu+2N+2)}.\]
A remainder identity expresses $Q'-Q_N'$ in terms of $p_N^2$ and yields bounds governed by $h_N$ (\cref{lem:gauss_remainder}).
By Stirling's formula, $\log h_N=-4N\log N+O_\nu(N)$, from which the compact transform estimate follows.
For the distributional estimate, \cref{lem:transfer} bounds the Kolmogorov distance between two Gamma convolutions by the transform
difference on an imaginary-axis interval and by the transforms at its endpoint. The remainder supplies the first on an interval of
length of order $N^2$, the decay of the transform of $T$ along the imaginary axis supplies the second, and balancing the two leads to
the uniform geometric bound in \cref{thm:intro_convergence}(ii).

Two further properties describe the approximants. Their total Gamma shape is
\[\sum_{i=1}^N\frac{\hat w_i}{\hat c_i}=N(N+\nu+1),\]
and $\phi_N(\lambda)$ decreases to $\phi(\lambda)$ for every $\lambda>0$.
Consequently, for each $q>0$, the negative moment $\E[\widehat T_N^{-q}]$ is finite precisely when $N(N+\nu+1)>q$,
and these moments decrease in $N$ to $\E[T^{-q}]<\infty$ (\cref{prop:shape,prop:negmom}).
The spectral construction also connects the approximation to tridiagonal eigenvalue methods for Bessel and Coulomb zeros \cite{GradZakrajsek,Ikebe}.
Here the full quadrature spectrum, together with its spectral weights, determines the moment-optimal approximation of the exit law (\cref{rem:ikebe}).

The normalization to radius $1$ is a convention rather than a restriction, since the exit time $T_a$ at radius $a>0$
satisfies $T_a\overset d=a^2T$ and has approximants $a^2\widehat T_N$ (\cref{prop:bridge}).
Throughout, we write $S_m=\sum_{k\ge1}c_k^m=2^m\sigma_m(\nu)$, where $\sigma_m(\nu)=\sum_{k\ge1}j_{\nu,k}^{-2m}$ are the Rayleigh
sums, for $m\ge1$. We use the standard $i=\sqrt{-1}$ for the imaginary unit, $\Re z$ denotes the real part of $z\in\C$, and complex powers and logarithms are taken on their principal branches.

The remainder of the paper unfolds as follows.
We begin in \cref{sec:scale} by identifying the scale measure and deriving the Jacobi recurrence that determines the approximants.
The construction and elementary node bounds established there provide the preparation for \cref{sec:approx}, where we prove
\cref{thm:intro_match} and then the two convergence estimates in \cref{thm:intro_convergence}.
With the main theorems established, we return in \cref{sec:properties} to additional features of the
construction, including its size-bias interpretation, total Gamma shape, endpoint behaviour, and monotone convergence of negative moments.
The numerical comparisons in \cref{sec:numerics} illustrate the resulting approximation alongside the mean-corrected truncation of \cref{eq:ct}.
Appendix~\ref{app:nodes} returns to the elementary bounds from \cref{sec:scale} and strengthens them to quantitative estimates
for the recovery of the largest atom and its mass. These estimates also yield relative accuracy of the upper-tail probabilities
on intervals whose lengths grow exponentially with $N$.

\section*{Disclosure of AI Tool Use} Claude Fable 5.1 (Anthropic) was used to write
Python code using \texttt{mpmath} for the high-precision calculations in \cref{sec:numerics}.
Claude was also used for editing the manuscript and polishing the exposition.
ChatGPT 6.1 Sol (OpenAI) assisted with proofreading and suggested several improvements that strengthened the paper's contributions.
ChatGPT also pointed out the connection of the paper's Gamma-Lommel polynomial approximation to the Gamma-Pad\'e approximation literature.
The author takes full responsibility for the mathematical content of this paper.

\section{The scale measure and its orthogonal polynomials}\label{sec:scale}

In this section we identify the scale measure with an even reduction of the modified-Lommel measure and derive the
Jacobi recurrence that determines its Gaussian quadrature rules.

\begin{lemma}\label{lem:allbeta}
For every $\nu>-1$ and every $\beta>0$ the measure $\Lom_{\nu,\beta}$ is finite and symmetric with
infinitely many atoms, and its monic orthogonal polynomials satisfy $P_0=1$, $P_1(x)=x$, and
\[P_{m+1}(x)=x\,P_m(x)-A_m^2\,P_{m-1}(x),\qquad m\ge1,\]
with $A_m^2$ as in \cref{eq:A2}.
\end{lemma}

\begin{proof}
By the Rayleigh identity $\sum_kj_{\nu,k}^{-2}=1/(4(\nu+1))$ \cite[\S15.51]{Watson}, $\Lom_{\nu,\beta}$ has total mass $\beta^2/(2(\nu+1))$.
At $\beta=1$, Dickinson's modified Lommel polynomials \cite{Dickinson} have monic normalization $P_m^{(1)}$ satisfying the recursion
\[P_{m+1}^{(1)}(x)=xP_m^{(1)}(x) - \frac{1}{4(\nu+m)(\nu+m+1)}P_{m-1}^{(1)}(x).\]
(See \cref{rem:lommel} below for the parameter identification.)
For every $\beta>0$, $\Lom_{\nu,\beta}$ is $\beta^2$ times the image of $\Lom_{\nu,1}$ under $x\mapsto\beta x$,
as can be seen by comparing atoms and masses. Its monic orthogonal polynomials are therefore
\[P_m^{(\beta)}(x)=\beta^mP_m^{(1)}(x/\beta).\]
Substitution into the unit-scale recurrence multiplies its last coefficient by $\beta^2$, giving \cref{eq:A2} for every positive $\beta$.
\end{proof}

\begin{remark}\label{rem:lommel}
For $\mu>0$, Watson's Lommel polynomials $R_{m,\mu}(z)$ are defined by the connection formula
\[J_{\mu+m}(z)=R_{m,\mu}(z)\,J_\mu(z)-R_{m-1,\mu+1}(z)\,J_{\mu-1}(z),\]
or equivalently by the recurrence
\[R_{m+1,\mu}(z)=(2(\mu+m)/z)\,R_{m,\mu}(z)-R_{m-1,\mu}(z),\]
with $R_{0,\mu}=1$ and $R_{1,\mu}(z)=2\mu/z$ \cite[\S9.6]{Watson}.
They are polynomials in $1/z$. Dickinson's modified polynomials $h_{m,\mu}(x):=R_{m,\mu}(1/x)$ are of degree $m$ in $x$ and satisfy the three-term recurrence
\[h_{m+1,\mu}(x)=2(\mu+m)\,x\,h_{m,\mu}(x)-h_{m-1,\mu}(x).\]
Dickinson \cite{Dickinson} identified the orthogonality measure of $\{h_{m,\mu}\}$ as the discrete
symmetric measure supported at the reciprocal zeros $\{\pm 1/j_{\mu-1,k}\}_{k\ge1}$; a textbook account
is \cite[Chapter~VI]{Chihara}. Taking $\mu=\nu+1>0$ and matching leading coefficients, the monic orthogonal polynomials of
$\Lom_{\nu,\beta}$ are scalar multiples of $h_{m,\nu+1}(x/\beta)$.
\end{remark}

\begin{proposition}\label{prop:bridge}
Let $\nu>-1$ and $\beta>0$.
The image of the restriction of $\Lom_{\nu,\beta}$ to $(0,\infty)$ under $x\mapsto2x^2/\beta^2$ is $(\beta^2/2)\,\rho_\nu$; at $\beta=\sqrt2$ the map is $x\mapsto x^2$ and the image is $\rho_\nu$.
For $a>0$, let $T_a$ be the exit time at radius $a$. Then $T_a\overset{d}{=}a^2T$.
The image of the restriction of $\Lom_{\nu,a\sqrt2}$ to $(0,\infty)$ under $x\mapsto x^2$ is $\sum_{k\ge1}a^2c_k\,\delta_{a^2c_k}$, the scale measure of $T_a$.
Moreover $\rho_\nu(\R)=S_1=\E[T]=1/(2\nu+2)$, and the cumulants of $T$ are
\begin{equation}\label{eq:cumulants}
\kappa_n(T)=(n-1)!\,S_n=(n-1)!\,2^n\sigma_n,\qquad
\sigma_n =
\begin{cases}
\frac{1}{4(\nu+1)}, & n=1,\\[2pt]
\frac{1}{n+\nu}\sum_{m=1}^{n-1}\sigma_m\sigma_{n-m}, & n\ge2.
\end{cases}
\end{equation}
\end{proposition}

\begin{proof}
The atom of $\Lom_{\nu,\beta}$ at $\beta/j_{\nu,k}$ has mass $(\beta/j_{\nu,k})^2=(\beta^2/2)c_k$ and
image point $2/j_{\nu,k}^2=c_k$, so the image measure is $(\beta^2/2)\rho_\nu$, which is $\rho_\nu$ at
$\beta=\sqrt2$. With $\beta=a\sqrt2$ and the map $x\mapsto x^2$, the atom at $a\sqrt2/j_{\nu,k}$ has
image point $a^2c_k$ and mass $a^2c_k$, and $T_a\overset d=a^2T$ follows by Brownian scaling.
The total mass is $2\sigma_1$, and $\E[T]=\sum_kc_k=S_1$ from \cref{eq:ct}.
Taking logarithms in \cref{eq:expectation_ct} and expanding $\log(1+\lambda c_k)$, we conclude for any $n\ge1$ that $\kappa_n(T)=(n-1)!S_n$,
and the recursion for $\sigma_n$ is that of \cite{Kishore}.
\end{proof}

\begin{proposition}\label{prop:jacobi}
Set $a_m^2=1/\bigl(2(\nu+m)(\nu+m+1)\bigr)$, the recurrence coefficients in \cref{eq:A2} at $\beta=\sqrt2$.
The monic orthogonal polynomials $p_n$ of $\rho_\nu$ satisfy $p_0(y)=1$, $p_1(y)=y-\alpha_0$, and
\begin{equation}\label{eq:fold}
p_{n+1}(y)=(y-\alpha_n)\,p_n(y)-\gamma_n\,p_{n-1}(y),\qquad n\ge1,
\end{equation}
with $\alpha_0=a_1^2$, $\alpha_n=a_{2n}^2+a_{2n+1}^2$, and $\gamma_n=a_{2n-1}^2a_{2n}^2$ for $n\ge1$;
all coefficients are rational functions of $\nu$.
\end{proposition}

\begin{proof}
By \cref{lem:allbeta} the monic orthogonal polynomials $P_m$ of $\Lom_{\nu,\sqrt2}$ satisfy the recurrence
\[P_{m+1}(x)=xP_m(x)-a_m^2P_{m-1}(x).\]
Observing that the measure is symmetric, we use the even reduction of a symmetric Jacobi recurrence to write $P_{2n}(x)=q_n(x^2)$ for monic $q_n$.
Eliminating the odd degrees folds the recurrence in $x$ into a recurrence in $x^2$,
\[P_{2n+2}(x)=\bigl(x^2-a_{2n}^2-a_{2n+1}^2\bigr)P_{2n}(x)-a_{2n-1}^2a_{2n}^2\,P_{2n-2}(x),\]
so the $q_n$ satisfy \cref{eq:fold} with respect to the image of $\Lom_{\nu,\sqrt2}$ under $x\mapsto x^2$,
which is $2\rho_\nu$ by \cref{prop:bridge} and the collapse of the $\pm$ atoms.
Rescaling a measure by a positive constant leaves its monic orthogonal polynomials unchanged, so it follows that $p_n=q_n$.
\end{proof}

Two tridiagonal matrix realizations of the recursion in \cref{eq:fold} are
\begin{equation}\label{eq:matrices}
H_N:=\begin{pmatrix}\alpha_0&1&&\\\gamma_1&\alpha_1&\ddots&\\&\ddots&\ddots&1\\&&\gamma_{N-1}&\alpha_{N-1}\end{pmatrix} \quad\text{and}\quad
J_N:=\begin{pmatrix}\alpha_0&\sqrt{\gamma_1}&&\\\sqrt{\gamma_1}&\alpha_1&\ddots&\\&\ddots&\ddots&\sqrt{\gamma_{N-1}}\\&&\sqrt{\gamma_{N-1}}&\alpha_{N-1}\end{pmatrix},
\end{equation}
the companion form and the symmetric Jacobi matrix of the recurrence, respectively.
The entries of $H_N$ are rational functions of $\nu$. Furthermore, conjugation by the diagonal matrix with entries
$(\gamma_1\cdots\gamma_n)^{-1/2}$, $0\le n\le N-1$, carries $H_N$ to $J_N$, so the two share the same characteristic polynomial, which is $p_N$.

The rational matrix $H_N$ is convenient for exact algebraic calculations, and the symmetric tridiagonal matrix $J_N$ for
floating-point eigenvalue computations. Through its spectral decomposition, $J_N$ also provides a finite-dimensional construction of the Gaussian quadrature rule
and hence of the Gamma approximant introduced in \cref{sec:intro}.

\begin{proposition}
\label{prop:quadrature}
For every $N\ge1$, the Gauss nodes $\hat c_1,\ldots,\hat c_N$ of $\rho_\nu$ are the eigenvalues of $J_N$, equivalently of $H_N$.
If $v_i$ is a real unit eigenvector of $J_N$ associated with $\hat c_i$, then the corresponding quadrature weight is
\begin{equation}\label{eq:quadrature_weights}
\hat w_i=S_1v_i(1)^2,\quad\text{with}\quad \sum_{i=1}^N\hat w_i=S_1.
\end{equation}
In particular, with $e_1=(1,0,\ldots,0)^\top$, one has
\begin{equation}\label{eq:jacobi_quadrature}
\int f(y)\,d\rho_\nu(y) = S_1 e_1^\top f(J_N)e_1 = \sum_{i=1}^N\hat w_i f(\hat c_i), \qquad \deg f\le2N-1.
\end{equation}
The Laplace transform of the resulting Gamma approximant is
\begin{equation}\label{eq:phiN}
\phi_N(\lambda)=\E[e^{-\lambda\widehat T_N}] = \prod_{i=1}^N(1+\lambda\hat c_i)^{-\hat w_i/\hat c_i}, \qquad \lambda\ge0.
\end{equation}
\end{proposition}

\begin{proof}
The characteristic polynomial of $J_N$ is $p_N$, whose zeros are the Gauss nodes.
We next identify the weights. For this purpose, let $\psi_m$ be the orthonormal polynomials of $\rho_\nu$ with positive leading
coefficients, so that $\psi_0=S_1^{-1/2}$, and form the matrix
\[V_{mi}=\sqrt{\hat w_i}\,\psi_{m-1}(\hat c_i), \qquad 1\le m,i\le N.\]
Since $\deg(\psi_{m-1}\psi_{n-1})\le2N-2$, by Gaussian exactness and orthonormality one has
\[(VV^\top)_{mn}=\sum_{i=1}^N\hat w_i\,\psi_{m-1}(\hat c_i)\,\psi_{n-1}(\hat c_i)=\int\psi_{m-1}\psi_{n-1}\,d\rho_\nu=\delta_{mn},\]
so that $V$ is orthogonal.
The orthonormal recurrence and $\psi_N(\hat c_i)=0$ show that the $i$th column of $V$ is an eigenvector of $J_N$ with eigenvalue $\hat c_i$.
These columns are orthonormal, and their first entries are $\sqrt{\hat w_i/S_1}$.
We have thus established the weight formula, the Jacobi-matrix realization of Gaussian quadrature \cite{GolubWelsch}.
\Cref{eq:jacobi_quadrature} then follows from the spectral decomposition of $J_N$, and taking $f=1$ recovers the total weight.
Finally, independence of the Gamma factors establishes \cref{eq:phiN}.
\end{proof}

\begin{remark}\label{rem:ikebe}
The tridiagonal construction in \cref{eq:matrices} is related, through the variable change $y=2/\rho^2$ and the folding of \cref{prop:jacobi},
to eigenvalue methods for Bessel zeros \cite{GradZakrajsek}; analogous methods for Coulomb zeros are developed in \cite{Ikebe}.
Here the full spectrum of $J_N$ and its spectral weights form a Gaussian quadrature rule.
The quadrature uses every node, including those that approximate individual atoms of $\rho_\nu$ less accurately.
Together, the nodes and weights determine the unique Gamma convolution with at most $N$ factors matching the first $2N$ cumulants of the exit law (\cref{thm:intro_match}).
\end{remark}

For the distributional estimate in \cref{thm:intro_convergence}(ii), we will use the Gamma factor associated with the largest quadrature node.
The following bounds keep its scale bounded away from zero uniformly in $N$ for fixed $\nu$ and show that its shape is greater than one for every $N$.

\begin{proposition}
\label{prop:nodes}
For every $N\ge1$, the quadrature nodes satisfy
\[0<\hat c_1<\cdots<\hat c_N<c_1,\]
and the largest node is strictly increasing with $N$. Moreover,
\begin{equation}\label{eq:node_bounds}
\hat c_N\ge a_1^2=\frac{1}{2(\nu+1)(\nu+2)}, \qquad \hat w_N\ge c_1,
\end{equation}
and consequently $\hat w_N/\hat c_N>1$.
\end{proposition}

\begin{proof}
The zeros of $p_N$ are real and simple, the zeros of consecutive orthogonal polynomials strictly interlace \cite{Chihara}, and all zeros lie in $(0,c_1)$.
Indeed, consider a zero $x$ of $p_N$ and put $q(y)=p_N(y)/(y-x)$. By orthogonality we have
\[0 = \int p_N(y)q(y)\,d\rho_\nu(y) = \int (y-x)q(y)^2\,d\rho_\nu(y),\]
whence
\[x=\frac{\int yq(y)^2\,d\rho_\nu(y)}{\int q(y)^2\,d\rho_\nu(y)}\in(0,c_1).\]
Here $x>0$ because every atom of $\rho_\nu$ is positive, and $x<c_1$ because $q$, having at most $N-1$ zeros, is nonzero at some atom $c_k$ with $k\ge2$.
Strict interlacing places the largest zero of $p_N$ strictly below the largest zero of $p_{N+1}$, so the largest node increases strictly from the $N$-point rule to the $(N+1)$-point rule.
The one-point rule has node $\alpha_0=a_1^2$. From this, we deduce the first bound in \cref{eq:node_bounds}.

For the weight bound, let
\[\ell_N(y)=\prod_{i=1}^{N-1} \frac{y-\hat c_i}{\hat c_N-\hat c_i},\]
with the empty product interpreted as $1$ when $N=1$.
Since $\deg\ell_N^2=2N-2$ and $\ell_N(\hat c_i)=\delta_{iN}$, by Gaussian exactness
\[\hat w_N=\int\ell_N(y)^2\,d\rho_\nu(y) \ge c_1\ell_N(c_1)^2\ge c_1,\]
with the last inequality following from $c_1>\hat c_N>\hat c_i$ for $i<N$.
\end{proof}

\section{Proofs of the main theorems}\label{sec:approx}

The proof of \cref{thm:intro_match} translates cumulant matching into a Gaussian quadrature problem.
An exact quadrature remainder then yields the transform estimate in \cref{thm:intro_convergence}(i), and a general inversion
inequality for Gamma convolutions turns it into the distributional estimate in (ii).

\subsection{Moment optimality}

\begin{proof}[Proof of \cref{thm:intro_match}.]
For a convolution of Gamma laws with shape $\alpha_i$ and scale $c_i$,
\[\kappa_n=(n-1)!\sum_i\alpha_ic_i^n=(n-1)!\int y^{n-1}\,d\mu(y),\]
with $\mu=\sum_i\alpha_ic_i\delta_{c_i}$, a positive measure with at most $N$ atoms.
The exit law has the same form with $\rho_\nu$ in place of $\mu$, by \cref{eq:cumulants}.
In other words, matching $\kappa_1,\dots,\kappa_{2N}$ amounts to matching the moments
of order $0,\dots,2N-1$ of $\rho_\nu$, and the Gauss measure $\sum_i\hat w_i\delta_{\hat c_i}$ matches by construction.
For uniqueness, it is enough to show that any matching measure $\mu$ is the Gauss measure.
Let $q$ be the monic node polynomial of $\mu$, of degree at most $N$.
If $\deg q\le N-1$, then $q^2$ has degree at most $2N-2$ and $\int q^2\,d\rho_\nu=\int q^2\,d\mu=0$.
But this is impossible because $\rho_\nu$ has infinitely many atoms, so $\mu$ must have precisely $N$ atoms.
For every polynomial $r$ of degree less than $N$, the product $qr$ has degree at most $2N-1$,
so $\int qr\,d\rho_\nu=\int qr\,d\mu=0$, hence $q=p_N$ and the atoms are the Gauss nodes.
The weights are then determined by the moments of order $0,\dots,N-1$ through an invertible Vandermonde system, so $\mu$ is the Gauss measure.
For the final assertion, matching $\kappa_1,\dots,\kappa_{2N+1}$ is matching the moments of order $0,\dots,2N$, and then
$\int q^2\,d\rho_\nu=\int q^2\,d\mu=0$ with $\deg q^2\le2N$, which is again impossible.
\end{proof}

The polynomial $y^{2N}-p_N(y)^2$ has degree at most $2N-1$, and, therefore, the mismatch at the next order equals
\[\kappa_{2N+1}(T)-\kappa_{2N+1}(\widehat T_N) = (2N)!\int p_N(y)^2\,d\rho_\nu(y)>0.\]
The same squared norm governs the quadrature remainder used below.

\subsection{The quadrature remainder}

Recall the slit plane $\mathcal D_\nu$ from the introduction.
For $z\in\mathcal D_\nu$, write
\[Q(z)=\sum_{k\ge1}\log(1+zc_k)\quad\text{and}\quad Q_N(z)=\sum_{i=1}^N\frac{\hat w_i}{\hat c_i}\log(1+z\hat c_i),\]
and set
\[h_N=\int p_N(y)^2\,d\rho_\nu(y).\]
The following lemma evaluates the squared norm $h_N$ in closed form and expresses the quadrature remainder through it.

\begin{lemma}\label{lem:gauss_remainder}
For every $N\ge1$,
\begin{equation}\label{eq:gauss_norm}
h_N=S_1\prod_{j=1}^N\gamma_j = \frac{\Gamma(\nu+1)^2}{2^{2N+1}\Gamma(\nu+2N+1)\Gamma(\nu+2N+2)}.
\end{equation}
The functions $Q$ and $Q_N$ are analytic on $\mathcal D_\nu$, and
\begin{equation}\label{eq:stieltjes_remainder}
Q'(z)-Q_N'(z) = \frac{z^{2N}}{\prod_{i=1}^N(1+z\hat c_i)^2}\int\frac{p_N(y)^2}{1+zy}\,d\rho_\nu(y),\qquad z\in\mathcal D_\nu.
\end{equation}
In particular,
\[0\le Q(\lambda)-Q_N(\lambda)\le\frac{h_N\lambda^{2N+1}}{2N+1},\qquad \lambda\ge0.\]
\end{lemma}

\begin{proof}
From the monic recurrence, $h_N=\gamma_Nh_{N-1}$ and $h_0=S_1$. Substituting the coefficients in \cref{prop:jacobi} yields
\[h_N=S_1\prod_{m=1}^{2N}a_m^2 = \frac1{2(\nu+1)}\prod_{m=1}^{2N}\frac1{2(\nu+m)(\nu+m+1)},\]
which is \cref{eq:gauss_norm}.

Each logarithm defining $Q$ is analytic on $\mathcal D_\nu$. It is not difficult to check that the series
converges locally uniformly there, because on any compact set its sufficiently late terms are bounded by a constant times $c_k$ and $\sum_kc_k=S_1<\infty$.
The nodes lie in $(0,c_1)$ by \cref{prop:nodes}, so $Q_N$ is analytic on the same domain. Differentiating we have
\[Q'(z)=\int\frac{d\rho_\nu(y)}{1+zy}\quad\text{and}\quad Q_N'(z)=\sum_{i=1}^N\frac{\hat w_i}{1+z\hat c_i}.\]
For $w\in\C\setminus[0,c_1]$, the expression
\[\frac{p_N(w)^2-p_N(y)^2}{w-y}\]
is a polynomial of degree $2N-1$ in $y$. By Gaussian exactness and $p_N(\hat c_i)=0$, it follows that
\[p_N(w)^2\int\frac{d\rho_\nu(y)}{w-y} - \int\frac{p_N(y)^2}{w-y}\,d\rho_\nu(y) = p_N(w)^2\sum_{i=1}^N\frac{\hat w_i}{w-\hat c_i}.\]
All zeros of $p_N$ lie in $(0,c_1)$, so division by $p_N(w)^2$ is valid. For $z\in\mathcal D_\nu$ with $z\ne0$, insert $w=-1/z$ and use
\[p_N(-1/z)=(-1/z)^N\prod_{i=1}^N(1+z\hat c_i).\]
Rearranging, one arrives at \cref{eq:stieltjes_remainder} for $z\ne0$. At $z=0$ both sides vanish, the left because
$Q'(0)=\rho_\nu(\R)=\sum_{i=1}^N\hat w_i=Q_N'(0)$ by \cref{eq:quadrature_weights}, and the right because of the factor $z^{2N}$.

It remains to show the final bound.
For $0\le t\le\lambda$ every factor on the right side of \cref{eq:stieltjes_remainder} is nonnegative.
Using the fact that $1+t\hat c_i\ge1$ and $1+ty\ge1$ on the support of $\rho_\nu$, we find
\[0\le Q'(t)-Q_N'(t)=\frac{t^{2N}}{\prod_{i=1}^N(1+t\hat c_i)^2}\int\frac{p_N(y)^2}{1+ty}\,d\rho_\nu(y)\le t^{2N}\int p_N(y)^2\,d\rho_\nu(y)=h_Nt^{2N}.\]
Since $Q(0)=Q_N(0)=0$, integrating over $0\le t\le\lambda$ we have
\[0\le Q(\lambda)-Q_N(\lambda)=\int_0^\lambda\bigl(Q'(t)-Q_N'(t)\bigr)\,dt\le h_N\int_0^\lambda t^{2N}\,dt=\frac{h_N\lambda^{2N+1}}{2N+1}.\qedhere\]
\end{proof}

\subsection{Convergence of the Laplace transforms}

\begin{proof}[Proof of \cref{thm:intro_convergence}(i).]
Analyticity follows from \cref{lem:gauss_remainder}, and the functions agree with the Laplace transforms on $\Re z>0$ by \cref{eq:expectation_ct,eq:phiN}.
The radial segments joining zero to the points of $K$ form a compact subset of $\mathcal D_\nu$. Consequently
\[L_K:=\sup_{z\in K}|z|<\infty\quad\text{and}\quad m_K:=\inf_{\substack{z\in K,\ 0\le u\le1\\0\le y\le c_1}}|1+uzy|>0.\]
The positivity of $m_K$ follows because $1+uzy=0$ would require $z\in(-\infty,-1/c_1]$.
Since $\mathcal D_\nu$ is star-shaped about $0$ and $Q(0)=Q_N(0)=0$,
\[Q(z)-Q_N(z)=\int_0^1\bigl(Q'-Q_N'\bigr)(uz)\,z\,du,\qquad z\in K.\]
In \cref{eq:stieltjes_remainder} at $uz$, the numerator is at most $u^{2N}L_K^{2N}$ in modulus, each factor
$|1+uz\hat c_i|$ and each $|1+uzy|$ on the support of $\rho_\nu$ is at least $m_K$,
and $\int p_N^2\,d\rho_\nu=h_N$. Hence
\[|(Q'-Q_N')(uz)|\le h_Nu^{2N}L_K^{2N}m_K^{-2N-1},\]
and integrating against $|z|\,du\le L_K\,du$ we have
\begin{equation}\label{eq:compact_remainder}
\sup_{z\in K}|Q(z)-Q_N(z)| \le \frac{h_N}{2N+1}\left(\frac{L_K}{m_K}\right)^{2N+1}.
\end{equation}
We apply Stirling's formula to \cref{eq:gauss_norm} and find $\log h_N=-4N\log N+O_\nu(N)$,
which proves the asserted bound for the logarithms. Finally,
\[|e^{-Q(z)}-e^{-Q_N(z)}|\le |Q(z)-Q_N(z)|\exp\bigl(\max\{-\Re Q(z),-\Re Q_N(z)\}\bigr).\]
Since $Q$ is continuous there and \cref{eq:compact_remainder} bounds $Q-Q_N$ uniformly, we thus see that the
exponential factor is uniformly bounded on $K$ and over $N$, which proves the transform bound.
\end{proof}

\subsection{Convergence of the distribution functions}

The distributional estimate rests on an inequality that does not involve a Bessel function.
It applies to any two convolutions of Gamma laws that each contain a factor of shape at least one, and it bounds
their Kolmogorov distance by the transform difference on an imaginary-axis interval and by the transforms at its endpoint.

\begin{lemma}\label{lem:transfer}
Let $X$ and $\widehat X$ be finite or convergent countable sums of independent Gamma variables, with distribution functions
$F$ and $\widehat F$ and transforms
\[\phi(z)=\E[e^{-zX}] \quad\text{and}\quad \widehat\phi(z)=\E[e^{-z\widehat X}].\]
Fix $H>0$, and suppose that each sum contains a summand whose shape is at least one and whose scale $c$ satisfies $Hc\ge1$.
If $|\widehat\phi(iu)-\phi(iu)|\le\varepsilon(u)$ for $0\le u\le H$, then
\[\sup_{x\in\R}|\widehat F(x)-F(x)|\le\frac1\pi\left[\int_0^H\varepsilon(u)\,\frac{du}{u}+2\sqrt2\,|\phi(iH)|+\sqrt2\,\varepsilon(H)\right].\]
\end{lemma}

\begin{proof}
Both distribution functions are continuous, since each law is a convolution containing a Gamma density.
By Fourier inversion for continuous distribution functions \cite{Feller},
\[\sup_{x\in\R}|\widehat F(x)-F(x)|\le\frac1\pi\int_0^\infty\frac{|\widehat\phi(iu)-\phi(iu)|}{u}\,du\]
whenever the right side is finite, and on $[0,H]$ the integrand is at most $\varepsilon(u)/u$.

Now we consider the frequencies beyond $H$. Write $\widehat X=\widehat G+\widehat Y$, where $\widehat G\sim\Gamma(\hat b,\hat c)$
is the distinguished summand and $\widehat Y$ is the sum of the others. The modulus $|\E[e^{-iu\widehat Y}]|$ is a product of
factors $(1+u^2c_i^2)^{-b_i/2}$ and is therefore nonincreasing in $u\ge0$. So for $u\ge H$
\[\frac{|\widehat\phi(iu)|}{|\widehat\phi(iH)|}\le\left(\frac{1+H^2\hat c^2}{1+u^2\hat c^2}\right)^{\hat b/2}
\le\left(\frac{1+H^2\hat c^2}{1+u^2\hat c^2}\right)^{1/2}\le\sqrt2\,\frac Hu,\]
where the second inequality uses $\hat b\ge1$ together with a ratio at most one, and the third uses $H\hat c\ge1$.
Integrating over $u\ge H$,
\[\int_H^\infty\frac{|\widehat\phi(iu)|}{u}\,du\le\sqrt2\,|\widehat\phi(iH)|,\]
and the same argument shows that $\int_H^\infty|\phi(iu)|\,du/u\le\sqrt2\,|\phi(iH)|$.
Therefore, using $|\widehat\phi(iH)|\le|\phi(iH)|+\varepsilon(H)$, the integral over $[H,\infty)$ is at most
$2\sqrt2\,|\phi(iH)|+\sqrt2\,\varepsilon(H)$, and adding the contribution of $[0,H]$ proves the lemma.
\end{proof}

\begin{proof}[Proof of \cref{thm:intro_convergence}(ii).]
We will deduce the estimate from \cref{lem:transfer}, applied to $X=T$ and $\widehat X=\widehat T_N$
with $H=\eta N^2$ for a constant $\eta>0$ to be chosen.
For $u\ge0$ and $y\ge0$, $|1+iuy|\ge1$, so from \cref{eq:stieltjes_remainder} and radial integration
\[|Q(iu)-Q_N(iu)|\le\frac{h_Nu^{2N+1}}{2N+1}.\]
Both $Q(iu)$ and $Q_N(iu)$ have nonnegative real part, so
\[|\phi(iu)-\phi_N(iu)|\le\varepsilon(u):=\frac{h_Nu^{2N+1}}{2N+1},\qquad u\ge0.\]
Here we are using the exponential difference estimate from the proof of (i). Integrating, we obtain
\begin{equation}\label{eq:cdf_low_frequency}
\int_0^H\varepsilon(u)\,\frac{du}{u}=\frac{\varepsilon(H)}{2N+1}.
\end{equation}
The exit law contains the exponential factor $c_1e_1\sim\Gamma(1,c_1)$ of \cref{eq:ct}, and $\widehat T_N$ contains the factor
$\Gamma(b_N,\hat c_N)$ attached to the largest node, with $b_N=\hat w_N/\hat c_N>1$ by \cref{prop:nodes}.
The same proposition supplies $\hat c_N\ge a_1^2$. Thus $H\hat c_N\ge1$ as long as $H\ge a_1^{-2}$, which holds for all
sufficiently large $N$, and $Hc_1\ge1$ follows from $c_1>\hat c_N$. We may therefore use \cref{lem:transfer}, which for such $N$ yields
\begin{equation}\label{eq:cdf_inversion_bound}
\sup_{x\in\R}|F_N(x)-F(x)| \le \frac1\pi\left[\frac{\varepsilon(H)}{2N+1}+\sqrt2\,\varepsilon(H)+2\sqrt2\,|\phi(iH)|\right].
\end{equation}

By Stirling's formula in \cref{eq:gauss_norm}, $\log h_N=-4N\log N+(4-6\log2)N+O_\nu(\log N)$, and hence
\[\varepsilon(\eta N^2) = O_{\nu,\eta}\!\left(N^{-2\nu-1}\left(\frac{e^4\eta^2}{64}\right)^N\right).\]
To estimate the exact transform at the cutoff, use
\[\phi(z)=\frac{(\sqrt{2z}/2)^\nu}{\Gamma(\nu+1)I_\nu(\sqrt{2z})}.\]
For $z=iu$ with $u>0$, the principal square root is $(1+i)\sqrt u$.
The large-argument asymptotic of $I_\nu$ on this ray \cite[\S7.23]{Watson} yields
\[|\phi(iu)|=O_\nu\!\left(u^{\nu/2+1/4}e^{-\sqrt u}\right)\quad\text{and}\quad |\phi(i\eta N^2)|=O_{\nu,\eta}\!\left(N^{\nu+1/2}e^{-\sqrt \eta\,N}\right).\]
It is now easy to finish the proof. We insert these estimates into \cref{eq:cdf_inversion_bound} and obtain, for every $\eta>0$,
\begin{equation}\label{eq:two_term}
\sup_{x\in\R}|F_N(x)-F(x)| = O_{\nu,\eta}\!\left(N^{-2\nu-1}\left(\frac{e^4\eta^2}{64}\right)^N + N^{\nu+1/2}e^{-\sqrt \eta\,N}\right).
\end{equation}
The two bases coincide when $e^4\eta^2/64=e^{-\sqrt \eta}$, which for $s=\sqrt \eta$ is the equation $s+4\log s=6\log2-4$ defining $s_*$.
Taking $\eta=s_*^2$, both exponential factors equal $q_*^N$, and the polynomial factors are absorbed into $q^N$ for every $q\in(q_*,1)$.
For a suitable choice of $C_{\nu,q}$, the bound extends to the finitely many remaining $N$, as required.
\end{proof}

\begin{remark}\label{rem:rate}
The proof establishes the two-term estimate \cref{eq:two_term} for every cutoff $H=\eta N^2$. The first term comes from the quadrature
remainder on the Fourier interval $0\le u\le\eta N^2$ and the second from the transform of $T$ at its endpoint.
The choice $\eta =1/4$ produces the bases $e^4/1024\approx0.053$ and $e^{-1/2}\approx0.607$, and $\eta=s_*^2\approx0.710$ balances them
at $q_*\approx0.430$, the smallest base this argument produces.
The observed errors in \cref{sec:numerics} fall by a factor near $0.06$ per unit increase in $N$.
The endpoint estimate, which uses the $e^{-\sqrt u}$ decay of the exact transform along the imaginary axis,
may contribute to the gap between the proved and observed rates.
For that matter, so may the quadrature remainder bound and the absolute-value estimates in Fourier inversion.
\end{remark}

\section{Further properties of the approximants}\label{sec:properties}

\subsection{Thorin measures and size biasing}\label{sec:thorin_bias}

The scale measure has a probabilistic interpretation through the independent increment associated with size biasing.
In presenting this interpretation, it is convenient to begin with its relation to the Thorin measure of the exit law.

Let us first recall that the Thorin measure of a nonnegative clock random variable $H$ with a generalized Gamma convolution law is the unique positive measure $U_H$ that,
in scale coordinates, represents its Laplace exponent as
\[-\log \E[e^{-\sigma H}] = d_H\sigma + \int_{(0,\infty)}\log(1+\sigma s)\,U_H(ds), \qquad \sigma\ge0,\]
where $d_H\ge0$ is a deterministic shift and $\int_{(0,\infty)}\log(1+s)\,U_H(ds)<\infty$.
An atom of mass $b$ at scale $c$ corresponds to an independent Gamma summand with shape $b$ and scale $c$.
The rate-coordinate Thorin measure is the pushforward of $U_H$ under $s\mapsto 1/s$.

Both $T$ and $\widehat T_N$ are generalized Gamma convolutions in the sense of Thorin \cite{Bondesson}, with zero deterministic shift.
The Thorin measure of $T$ has a unit atom at each scale $c_k$, corresponding to the rate $1/c_k=j_{\nu,k}^2/2$.
Thus, in scale coordinates, the Thorin measure is
\[U := \sum_{k\ge1}\delta_{c_k}, \qquad\text{with}\qquad \rho_\nu(ds)=s\,U(ds).\]
The measure $U$ has infinitely many unit atoms, and its scale-weighted version $\rho_\nu$ has finite mass because
\[\rho_\nu((0,\infty))=\sum_{k\ge1}c_k = \E[T]<\infty.\]
The corresponding measures for the approximant are
\[\widehat U_N := \sum_{i=1}^N\frac{\hat w_i}{\hat c_i}\delta_{\hat c_i},\qquad \widehat\rho_N(ds) := s\,\widehat U_N(ds) = \sum_{i=1}^N\hat w_i\delta_{\hat c_i}.\]
Thus \cref{prop:quadrature} constructs the finite Thorin measure $\widehat U_N$ by applying Gaussian quadrature to the scale-weighted
measure $\rho_\nu$ and dividing each quadrature weight by its node.
The interpretation of Gamma convolution approximation as a truncated moment problem for a transformed Thorin measure is also developed in \cite[Section~2.3]{LavernyEtAl}.

We mention that, for a scaled sum of two independent three-dimensional Bessel exit times,
Polson \cite[Section~5]{Polson} considers a related Thorin-measure
approximation that retains the leading atoms and replaces the remainder by a continuous density.

One says that $T^*$ has the size-biased law of $T$ when
\[\E[g(T^*)]=\frac{\E[Tg(T)]}{S_1}\]
for bounded measurable $g$. Differentiating \cref{eq:expectation_ct}, we find
\[\E[e^{-\lambda T^*}] = -\frac{\phi'(\lambda)}{S_1} = \phi(\lambda)\frac1{S_1}\int\frac{\rho_\nu(ds)}{1+\lambda s}, \qquad \lambda\ge0.\]
Let $C$ have distribution $\rho_\nu/S_1$, let $e$ be an independent exponential random variable of mean one, and take both independent of $T$.
Setting $M=Ce$, we obtain
\begin{equation}\label{eq:size_bias_increment}
T^*\overset d=T+M\quad\text{and}\quad \E[e^{-\lambda M}] = \frac1{S_1}\int\frac{\rho_\nu(ds)}{1+\lambda s}.
\end{equation}
This means that $\rho_\nu/S_1$ is the distribution of the random scale in this exponential mixture.

Alternatively, the decomposition follows by size-biasing one summand in \cref{eq:ct}, chosen with probability proportional to its mean \cite[Section~2.4]{ArratiaGoldsteinKochman}.
Indeed, it is enough to observe that the summand of scale $c_k$ is chosen with probability $c_k/S_1$ and that its size-biased law is $\Gamma(2,c_k)$,
obtained by adding an independent exponential variable of scale $c_k$.
Choosing the summand independently of the original variables produces the increment $M$ above.

For a nonnegative random variable with finite positive mean, the existence of an independent nonnegative size-bias increment
characterizes infinite divisibility \cite[Theorem~51.1]{Sato}\cite[Theorem~11.1]{ArratiaGoldsteinKochman}.
Here the increment law can be identified from the L\'evy measure.
The exit law has zero deterministic drift and L\'evy measure
\[\Pi(dx)=\frac1x\int e^{-x/s}\,U(ds)\,dx, \qquad x>0,\]
as follows from its Gamma convolution representation
\cite{Bondesson}. Since $\int x\,\Pi(dx)=S_1$, the law of $M$ is the normalized size-biased L\'evy measure,
\begin{equation}\label{eq:size_bias_levy}
\frac{x\,\Pi(dx)}{S_1} = \left(\int\frac1s e^{-x/s}\,\frac{\rho_\nu(ds)}{S_1}\right)dx, \qquad x>0.
\end{equation}
Thus scale weighting of the Thorin measure determines the mixing law, while size weighting of the L\'evy measure determines the increment law.

For $\widehat T_N$, one has the analogous decomposition
\[\widehat T_N^*\overset d=\widehat T_N+\widehat M_N\quad\text{and}\quad \E[e^{-\lambda\widehat M_N}] = \frac1{S_1}\sum_{i=1}^N\frac{\hat w_i}{1+\lambda\hat c_i},\]
with $\widehat M_N$ independent of $\widehat T_N$ and its exponential scale drawn from $\widehat\rho_N/S_1$.
For every integer $m\ge0$,
\[\E[M^m]=\frac{m!}{S_1}\int s^m\,\rho_\nu(ds)\quad\text{and}\qquad\E[\widehat M_N^m] = \frac{m!}{S_1}\int s^m\,\widehat\rho_N(ds).\]
Consequently, from \cref{eq:cumulants} and the Gamma cumulant formula it can be deduced that
\[\kappa_n(T)=S_1\E[M^{n-1}]\quad\text{and}\quad \kappa_n(\widehat T_N)=S_1\E[\widehat M_N^{n-1}],\qquad n\ge1.\]
The foregoing establishes that the cumulant matching in \cref{thm:intro_match} is equivalent to
\[\E[\widehat M_N^m]=\E[M^m],\qquad 0\le m\le2N-1.\]
The Gaussian rule preserves these moments by discretizing the random scale of the size-bias increment.
In this sense, cumulant matching for the convolution is moment matching for an exponential mixture.

\subsection{Total Gamma shape and endpoint behaviour}

\begin{proposition}\label{prop:shape}
For every $\nu>-1$ and $N\ge1$,
\begin{equation}\label{eq:shape}
\Sigma_N:=\sum_{i=1}^{N}\frac{\hat w_i}{\hat c_i}=N(N+\nu+1);
\end{equation}
although $\int y^{-1}\,d\rho_\nu$ diverges, the Gauss rule applied to $y\mapsto y^{-1}$ takes the finite value $N(N+\nu+1)$.
\end{proposition}

\begin{proof}
Write $a_m>0$ for the positive square root of the coefficient $a_m^2$, and let $B_N$ be the lower bidiagonal matrix with
\[(B_N)_{jj}=a_{2j-1},\quad 1\le j\le N,\quad\text{and}\quad (B_N)_{j+1,j}=a_{2j},\quad 1\le j<N.\]
It is straightforward to verify from the coefficients in \cref{prop:jacobi} that $J_N=B_NB_N^\top$,
and the spectral representation in \cref{prop:quadrature} then yields
\[\Sigma_N=S_1e_1^\top J_N^{-1}e_1 = S_1\|B_N^{-1}e_1\|^2.\]
Solving $B_Nx=e_1$ successively, we find
\[x_{n+1} = \frac{(-1)^n}{a_{2n+1}} \prod_{m=1}^n\frac{a_{2m}}{a_{2m-1}}, \qquad 0\le n<N,\]
with the empty product equal to $1$.
Since
\[\prod_{m=1}^n\frac{a_{2m}^2}{a_{2m-1}^2} = \prod_{m=1}^n\frac{\nu+2m-1}{\nu+2m+1} = \frac{\nu+1}{\nu+2n+1},\]
we obtain, making use of $a_{2n+1}^{-2}=2(\nu+2n+1)(\nu+2n+2)$ and $S_1=1/(2\nu+2)$,
\[\Sigma_N = S_1\sum_{n=0}^{N-1}\frac{1}{a_{2n+1}^2}\prod_{m=1}^n\frac{a_{2m}^2}{a_{2m-1}^2} = \sum_{n=0}^{N-1}(\nu+2n+2) = N(N+\nu+1)\]
as claimed.
\end{proof}

\begin{remark}\label{rem:shape}
$\Sigma_N$ is the total shape of the Gamma convolution $\widehat T_N$ and the total mass of its Thorin
measure in the rate variable, namely, $\sum_i(\hat w_i/\hat c_i)\,\delta_{1/\hat c_i}$.
Here the quadrature replaces the infinite-mass Thorin measure of $T$ of \cref{sec:thorin_bias} by one of mass $N(N+\nu+1)$.
Writing $b_i=\hat w_i/\hat c_i$ and $f_N$ for the density of $\widehat T_N$, convolution of the Gamma densities shows that
\[
f_N(t)\sim\frac{\prod_{i=1}^N\hat c_i^{-b_i}}{\Gamma(\Sigma_N)}\,t^{\Sigma_N-1},
\qquad t\downarrow0.
\]
Thus the exponent at the origin is $\Sigma_N-1=N(N+\nu+1)-1$, and
\[
F_N(t)\sim\frac{\prod_{i=1}^N\hat c_i^{-b_i}}{\Gamma(\Sigma_N+1)}\,t^{\Sigma_N},
\qquad t\downarrow0.
\]
By comparison, the exact CDF satisfies $\lim_{t\downarrow0}t\log F(t)=-1/2$.
Indeed, $Y_t/(2t)\sim\Gamma(\nu+1,1)$ \cite[Chapter~XI]{RevuzYor} and
$\{Y_t\ge1\}\subseteq\{T\le t\}$ furnish a lower bound of order
$t^{-\nu}e^{-1/(2t)}$ for $F(t)$; the upper bound
$F(t)\le e^{\lambda t}\phi(\lambda)$ at $\lambda=1/(2t^2)$ is of order
$t^{-\nu-1/2}e^{-1/(2t)}$ by the large-argument asymptotic of $I_\nu$
\cite[\S7.23]{Watson}. Consequently, $F_N(t)/F(t)\to\infty$ as $t\downarrow0$
for every fixed $N$.
\end{remark}

\begin{remark}\label{rem:tails}
Write $\bar F=1-F$ and $\bar F_N=1-F_N$. From the exponential representation
\cref{eq:ct},
\begin{equation}\label{eq:exact_tail}
\bar F(t)\sim A_\nu e^{-t/c_1},
\qquad
A_\nu:=\prod_{k\ge2}(1-c_k/c_1)^{-1}<\infty,
\qquad t\to\infty.
\end{equation}
To see this, write $T=c_1e_1+R$, where $R=\sum_{k\ge2}c_ke_k$ is independent of $e_1$.
For any $\tau\in(1/c_1,1/c_2)$ we have that $\E[e^{\tau R}]<\infty$, since $\sum_kc_k<\infty$.
Conditioning on $R$, we obtain
\[e^{t/c_1}\bar F(t) = \E[e^{R/c_1}\bigm| R\le t]+e^{t/c_1}\mathbb P(R>t)\longrightarrow \E[e^{R/c_1}]=A_\nu.\]
For the approximant, write $\widehat T_N=G_N+\widehat R_N$, with $G_N\sim\Gamma(b_N,\hat c_N)$ independent of the remaining factors.
Conditioning on $\widehat R_N$ and using the Gamma survival asymptotic, we obtain similarly
\begin{equation}\label{eq:gamma_tail}
\bar F_N(t)\sim B_Nt^{b_N-1}e^{-t/\hat c_N},\quad\text{with}\quad B_N:=\frac{\prod_{i<N}(1-\hat c_i/\hat c_N)^{-b_i}}
{\Gamma(b_N)\hat c_N^{\,b_N-1}}, \qquad t\to\infty.
\end{equation}
Empty products are interpreted as one. Since $\hat c_N<c_1$,
$\bar F_N(t)/\bar F(t)\to0$ for every fixed $N$.
Thus the relative error $|F_N(t)-F(t)|/\bar F(t)$ tends to one,
while the underestimation factor $\bar F(t)/\bar F_N(t)$ is unbounded.
Together with \cref{rem:shape}, this shows why a small Kolmogorov error alone
does not ensure relative accuracy for arbitrarily small tail probabilities or extreme quantiles.
Nonetheless, \cref{cor:relative_tail} establishes uniform relative accuracy of the upper tail on
$0\le t\le t_N$ whenever $t_N=o(r^{2N})$, with $r>1$ as in \cref{prop:extreme_rates}.
\end{remark}

\subsection{Monotonicity and negative moments}

For $\lambda>0$ write $g_\lambda(y)=\log(1+\lambda y)/y$, with $g_\lambda(0)=\lambda$, so that
$Q(\lambda)=\int g_\lambda\,d\rho_\nu$ and $Q_N(\lambda)$ is the $N$-point Gauss rule applied to $g_\lambda$;
by \cref{eq:expectation_ct,eq:phiN}, $\E[e^{-\lambda T}]=e^{-Q(\lambda)}$ and $\E[e^{-\lambda\widehat T_N}]=e^{-Q_N(\lambda)}$.
This allows us to compare the transforms of the approximants across $N$.

\begin{lemma}\label{lem:monotone}
For every $\lambda>0$,
\[Q_1(\lambda)\le Q_2(\lambda)\le\cdots\le Q(\lambda),\qquad Q_N(\lambda)\uparrow Q(\lambda),\]
so the transforms $\E[e^{-\lambda\widehat T_N}]$ decrease in $N$ to $\E[e^{-\lambda T}]$.
\end{lemma}

\begin{proof}
The identity $\int_{1/\lambda}^{\infty}\bigl(w(w+y)\bigr)^{-1}dw=\log(1+\lambda y)/y$ represents both sums as mixtures of resolvents,
\[Q_N(\lambda)=\int_{1/\lambda}^{\infty}G_N(w)\,\frac{dw}{w}\quad\text{and}\quad Q(\lambda)=\int_{1/\lambda}^{\infty}G(w)\,\frac{dw}{w},\]
with $G_N(w)=\sum_i\hat w_i/(w+\hat c_i)$ and $G(w)=\int(w+y)^{-1}d\rho_\nu(y)$.
The mixing measure $dw/w$ is positive and does not depend on $N$, so for the monotonicity $Q_N(\lambda)\le Q_{N+1}(\lambda)$ it suffices to show that $G_N(w)$ is nondecreasing in $N$ for each $w>0$.
By the spectral decomposition of \cref{prop:quadrature}, $G_N(w)=S_1\bigl[(J_N+wI)^{-1}\bigr]_{11}$, with $J_N$ the
symmetric matrix of \cref{eq:matrices}. Here we may invoke the variational formula
$(A^{-1})_{11}=\max_x\bigl(2x_1-x^\top Ax\bigr)$ for positive definite $A$.
Thanks to \cref{prop:nodes}, the eigenvalues of $J_{N+1}$, the nodes of the $(N+1)$-point rule, are positive,
so $J_{N+1}+wI$ is positive definite. Its leading principal block is $J_N+wI$,
and restricting the maximum to vectors with last coordinate zero shows that
$\bigl[(J_{N+1}+wI)^{-1}\bigr]_{11}\ge\bigl[(J_N+wI)^{-1}\bigr]_{11}$; multiplying by $S_1$, we obtain $G_{N+1}(w)\ge G_N(w)$.
Moreover, by the final bound of \cref{lem:gauss_remainder} and \cref{eq:gauss_norm}, $Q_N(\lambda)\to Q(\lambda)$ for every $\lambda>0$.
Thus $Q_N(\lambda)\uparrow Q(\lambda)$, and the asserted monotonicity of the Laplace transforms follows by exponentiation.
\end{proof}

Since negative moments are integrals of the Laplace transform against a power of $\lambda$, \cref{lem:monotone} has the following consequence for them.

\begin{proposition}\label{prop:negmom}
For every $q>0$ the negative moments $\E[\widehat T_N^{-q}]\in(0,\infty]$ are nonincreasing in $N$,
finite precisely when $q<\Sigma_N$, and
\[\lim_{N\to\infty}\E[\widehat T_N^{-q}] = \E[T^{-q}]<\infty.\]
\end{proposition}

\begin{proof}
For a positive random variable $X$ and $q>0$, by Tonelli's theorem we have
\[\Gamma(q)\,\E[X^{-q}]=\int_0^\infty\lambda^{q-1}\,\E[e^{-\lambda X}]\,d\lambda,\]
with both sides in $[0,\infty]$.
With $X=\widehat T_N$ the integrand is $\lambda^{q-1}e^{-Q_N(\lambda)}$.
Note then that $Q_N(\lambda)\ge\Sigma_N\log(1+\lambda\hat c_1)$, since $\hat c_1$ is the smallest node, so the integral is finite for $q<\Sigma_N$.
On the other hand, $e^{-Q_N(\lambda)}\ge(1+\lambda\hat c_N)^{-\Sigma_N}$, which shows divergence for $q\ge\Sigma_N$. By
\cref{lem:monotone} the integrands are nonincreasing in $N$ and decrease pointwise to
$\lambda^{q-1}e^{-Q(\lambda)}$, so the moments are nonincreasing. Fix $N_0$ with $\Sigma_{N_0}>q$.
The integrand at $N_0$ is integrable and dominates those for $N\ge N_0$,
and the result follows after an application of dominated convergence.
\end{proof}

\begin{remark}\label{rem:fredholm}
It is worthwhile to mention an operator interpretation of \cref{lem:monotone}.
One denotes by $J$ the Jacobi operator on $\ell^2$ with diagonal $(\alpha_n)$ and off-diagonal $(\sqrt{\gamma_n})$,
the semi-infinite extension of $J_N$ in \cref{eq:matrices}, compact and selfadjoint with simple eigenvalues $c_k$ and spectral measure $\rho_\nu/S_1$ at $e_1$.
Let $G_\nu$ be the Green operator of the Bessel process killed at $1$, the inverse of the selfadjoint realization of
\[-\tfrac12\bigl(\partial_x^2+(2\nu+1)x^{-1}\partial_x\bigr)\quad\text{on}\quad L^2\bigl((0,1),x^{2\nu+1}dx\bigr)\]
with the Dirichlet condition $f(1)=0$ and, for $-1<\nu<1$, the zero-flux boundary condition
\[\lim_{x\downarrow0}x^{2\nu+1}f'(x)=0.\]
For $\nu\ge1$ the origin is a limit-point endpoint.
The eigenfunctions of $G_\nu$ are $x^{-\nu}J_\nu(j_{\nu,k}x)$ with simple eigenvalues $c_k$, so $G_\nu$ is compact and selfadjoint.
The common simple eigenvalue sequence makes $J$ and $G_\nu$ unitarily equivalent.
\Cref{eq:expectation_ct} is the Fredholm determinant identity
\[\E[e^{-\lambda T}]=\det(I+\lambda G_\nu)^{-1},\]
the operator being trace class with $\operatorname{Tr}G_\nu=S_1$.
In this identity the Thorin measure $U=\sum_{k\ge1}\delta_{c_k}$ of \cref{sec:thorin_bias} is the eigenvalue counting measure of $G_\nu$,
an instance of the connection between the canonical measure of a diffusion hitting time
and the spectral measure of its generator developed in \cite{Kent1982}.
The scale-weighted spectral measure converts the trace into a matrix element,
\[Q(\lambda)=S_1\,\langle e_1,\,g_\lambda(J)\,e_1\rangle\quad\text{and}\quad Q_N(\lambda)=S_1\,\langle e_1,\,g_\lambda(J_N)\,e_1\rangle,\]
and the second expression evaluates $g_\lambda$ on the compression of $J$ to the order-$N$ Krylov subspace of $e_1$.
The resulting transform $e^{-Q_N}$ is a weighted spectral determinant, in which the quadrature assigns Gamma multiplicity $\hat w_i/\hat c_i$ to each effective
scale $\hat c_i$. The result is a Gauss approximation of the reciprocal Fredholm determinant through its scale-weighted spectral measure.
\end{remark}

\section{Numerical experiments}\label{sec:numerics}

The nodes and weights are constructed from the rational Jacobi data \cref{eq:fold}, and the exact power sums \cref{eq:cumulants} provide checks.
The construction requires no Bessel zeros or Bessel function values.
The symmetric tridiagonal matrix $J_N$ of \cref{eq:matrices} was diagonalized at $45$ decimal digits by the implicit QL algorithm,
and the weights were obtained from \cref{prop:quadrature}.

Cumulant matching in \cref{thm:intro_match} was verified in exact rational arithmetic.
The quantity compared is the normalized cumulant
\[\widehat S_n := \frac{\kappa_n(\widehat T_N)}{(n-1)!} = \sum_{i=1}^N\hat w_i\hat c_i^{\,n-1} = S_1\bigl(H_N^{\,n-1}\bigr)_{11},\]
where the last equality holds for every $n\ge1$ by the spectral decomposition of $J_N$ in \cref{prop:quadrature} and the diagonal
conjugation carrying $H_N$ to $J_N$.
For the indices $\nu=\tfrac12$, $0$, and $1.4=\tfrac75$ used below, the entries of $H_N$, the power sums $S_n=2^n\sigma_n$ computed by
Kishore's recursion, and the squared norm $h_N$ of \cref{eq:gauss_norm} are rational numbers. Thus the identities
\[\widehat S_n=S_n,\quad 1\le n\le2N,\quad\text{and}\quad \widehat S_{2N+1}-S_{2N+1}=-h_N \]
can be tested without rounding; both were verified for $1\le N\le12$.
In terms of the raw cumulants, the second identity reads
\[\kappa_{2N+1}(\widehat T_N)-\kappa_{2N+1}(T)=-(2N)!\,h_N.\]
The sums formed from the floating-point nodes and weights reproduce
the exact $\widehat S_n$ for $1\le n\le2N+1$ with relative errors below $3\times10^{-45}$.
At $\nu=0$ and $N=12$, direct floating-point subtraction cannot reliably resolve the mismatch at this precision, since
$h_{12}/S_{25}\approx1.05\times10^{-45}$ places it at the last retained digit of the two power sums.

Write $D_N(t)=F_N(t)-F(t)$. We evaluated $D_N$ by Talbot inversion of the difference transform $(\phi_N(z)-\phi(z))/z$,
thereby avoiding the final subtraction of two separately inverted distribution functions.
Each reported extremal value was recomputed with a higher Talbot degree and with the de Hoog algorithm, and the three evaluations
agree to at least twelve significant digits.
At seven grid points spread across the interval, raising the Talbot degree from $50$ to $70$ changed the computed values by at most $3\times10^{-31}$, providing an empirical check on inversion accuracy.
At $\nu=\tfrac12$, inversion of $\phi(z)/z$ was also compared with the theta series for $F$, with agreement to better than $10^{-34}$.

We searched adaptively for extrema of $|D_N|$ on the interval $t_-\le t\le t_+$, where $t_-=10^{-3}\,\E[T]$ and $t_+=60\,\E[T]$.
A logarithmic grid of $320$ points ($640$ for $N\ge8$) was placed on the interval, and each candidate local maximum identified on the
grid was refined by golden-section search.
Let $E_N^{\mathrm{num}}$ denote the largest refined value found, and let $t_*$ be its location.

Outside the search interval, monotonicity of $F$ and $F_N$ gives
\begin{equation}\label{eq:endpoint_bounds}
\begin{aligned}
\sup_{0<t\le t_-}|D_N(t)| &\le\max\{F_N(t_-),F(t_-)\},\\
\sup_{t\ge t_+}|D_N(t)| &\le\max\{1-F_N(t_+),1-F(t_+)\}.
\end{aligned}
\end{equation}
The four probabilities were bounded using the Chernoff inequalities
\begin{equation}\label{eq:chernoff}
F_X(t)\le e^{\lambda t}\,\E\bigl[e^{-\lambda X}\bigr]\quad\text{and}\quad
1-F_X(t)\le e^{-\lambda t}\,\E\bigl[e^{\lambda X}\bigr],
\end{equation}
valid for $\lambda>0$, with the second requiring $\lambda<1/c_1$ when $X=T$ and $\lambda<1/\hat c_N$ when $X=\widehat T_N$.
The exponential moments are $\phi(-\lambda)=1/{}_0F_1(;\nu+1;-\lambda/2)$ from
\cref{eq:expectation_ct} and the product \cref{eq:phiN} evaluated at $-\lambda$, respectively.
Each bound was minimized numerically over its admissible range.
For each pair $(\nu,N)$ in \cref{tab:sweep}, the larger of the two endpoint bounds is less than $10^{-11}E_N^{\mathrm{num}}$.
Consequently, any value of $|D_N|$ exceeding $E_N^{\mathrm{num}}$ must lie inside the search interval.
The table reports the largest extremum resolved by the interior search.

\begin{table}[ht]
\caption{Largest resolved extremum $E_N^{\mathrm{num}}$ of $|F_N(t)-F(t)|$ on $10^{-3}\,\E[T]\le t\le60\,\E[T]$,
and its location $t_*$ in units of $\E[T]=1/(2\nu+2)$.
For each parameter pair, the endpoint bounds in \cref{eq:endpoint_bounds} are smaller than the corresponding reported
extremum by at least eleven orders of magnitude.}
\label{tab:sweep}
\begin{tabular}{rcccccc}
\hline
 & \multicolumn{2}{c}{$\nu=1/2$ ($d=3$)}
 & \multicolumn{2}{c}{$\nu=0$ ($d=2$)}
 & \multicolumn{2}{c}{$\nu=1.4$ ($d=4.8$)}\\
$N$
 & $E_N^{\mathrm{num}}$ & $t_*/\E[T]$
 & $E_N^{\mathrm{num}}$ & $t_*/\E[T]$
 & $E_N^{\mathrm{num}}$ & $t_*/\E[T]$\\
\hline
 2 & $3.9\times10^{-3}$ & $0.52$
   & $3.3\times10^{-3}$ & $0.19$
   & $4.4\times10^{-3}$ & $0.60$\\
 4 & $1.8\times10^{-5}$ & $0.24$
   & $1.4\times10^{-5}$ & $0.18$
   & $2.3\times10^{-5}$ & $0.31$\\
 6 & $6.2\times10^{-8}$ & $0.21$
   & $3.8\times10^{-8}$ & $0.16$
   & $9.9\times10^{-8}$ & $0.29$\\
 8 & $2.1\times10^{-10}$ & $0.14$
   & $1.2\times10^{-10}$ & $0.10$
   & $3.7\times10^{-10}$ & $0.20$\\
10 & $7.1\times10^{-13}$ & $0.13$
   & $3.7\times10^{-13}$ & $0.095$
   & $1.5\times10^{-12}$ & $0.19$\\
12 & $2.2\times10^{-15}$ & $0.12$
   & $1.1\times10^{-15}$ & $0.072$
   & $5.5\times10^{-15}$ & $0.18$\\
\hline
\end{tabular}
\end{table}

At $\nu=\tfrac12$, $E_N^{\mathrm{num}}$ falls by a factor near $0.06$ per unit increase in $N$, as noted in \cref{rem:rate}.
Its location moves toward the origin as $N$ grows, from $0.52\,\E[T]$ at $N=2$ to $0.12\,\E[T]$ at $N=12$.
At $N=12$, values of $|D_N|$ exceeding $10^{-29}$ were resolved only for $0.02\,\E[T]<t<1.7\,\E[T]$.
At the mean, the signed errors $D_{12}(\E[T])$ are approximately $-1.4\times10^{-24}$, $7.0\times10^{-28}$, and $7.5\times10^{-22}$
for $\nu=\tfrac12,0,1.4$, respectively.

Moment matching also forces oscillation in the error curve.
By integration by parts, $\int_0^\infty t^jD_N(t)\,dt=0$ for $0\le j\le2N-1$,
so $D_N$ changes sign at least $2N$ times.
The numerical search resolved $2N$ sign changes for $N\le6$.
For larger $N$, some of the corresponding lobes have computed heights below $10^{-29}$.
These oscillations motivate the logarithmic grid and local refinement used in the search.

For comparison, the mean-corrected truncation
\[\sum_{k=1}^{N}c_ke_k+\Bigl(S_1-\sum_{k=1}^{N}c_k\Bigr)\]
at $\nu=\tfrac12$, examined by the same search procedure, has largest resolved errors
$2.9\times10^{-3}$, $9.4\times10^{-4}$, $4.2\times10^{-4}$,
$2.2\times10^{-4}$, and $1.3\times10^{-4}$ for $N=4,6,8,10,12$, respectively.
These values are consistent with a decay near $N^{-2.8}$ over the tested range.
At $N=12$, the largest resolved error is more than ten orders of magnitude larger than that of the Gauss construction.

For the same index $\nu=\tfrac12$, the largest quadrature node approaches $c_1=2/\sjz^2$ from below.
The differences $c_1-\hat c_N^{(N)}$ are approximately $1.5\times10^{-3}$, $9.4\times10^{-9}$, $1.4\times10^{-15}$,
and $1.7\times10^{-23}$ at $N=2,4,6,8$, respectively, and $3.3\times10^{-32}$ at $N=10$.
Thus $\sqrt{2/\hat c_{10}^{(10)}}$ reproduces $j_{1/2,1}=\pi$ to thirty digits as a byproduct.

\appendix
\section{Extreme parameters and relative tail accuracy}\label{app:nodes}

\Cref{prop:nodes} provides the node bounds used in the main proofs.
The estimates below describe how the largest quadrature node and its weight recover the isolated atom at $c_1$,
and how this recovery controls relative errors in the upper-tail probabilities.

\begin{proposition}
\label{prop:extreme_rates}
Fix $\nu>-1$, and set
\[u=\frac{2c_1-c_2}{c_2}\quad\text{and}\quad r=u+\sqrt{u^2-1}>1.\]
There are a constant $C_\nu>0$ and an integer $N_\nu\ge1$, depending only on $\nu$, such that
\begin{equation}\label{eq:noderate}
\begin{aligned}
0<c_1-\hat c_N&\le4S_1r^{-2(N-1)}, &&N\ge1,\\
0\le\hat w_N-c_1&\le C_\nu N r^{-2N}, &&N\ge N_\nu.
\end{aligned}
\end{equation}
Thus the largest node and its weight converge geometrically to the location and mass of the atom $\rho_\nu(\{c_1\})=c_1$.
\end{proposition}

\begin{proof}
Let $T_m$ denote the Chebyshev polynomial of the first kind, and set
\[\varphi(y)=\frac{2y-c_2}{c_2}\quad\text{and}\quad \pi_N(y)=\frac{T_{N-1}(\varphi(y))}{T_{N-1}(u)}.\]
Observe that $\pi_N(c_1)=1$. Since $|T_{N-1}|\le1$ on $[-1,1]$ and $T_{N-1}(u)\ge\tfrac12r^{N-1}$,
\[\sup_{0\le y\le c_2}|\pi_N(y)|\le2r^{-(N-1)}.\]
We apply the Gauss rule to $\pi_N^2$ and to $(c_1-y)\pi_N(y)^2$, for both of which it is exact, their degrees being at most $2N-2$ and $2N-1$.
The atom at $c_1$ contributes at least $c_1$ to the integral, so
\[\sum_{i=1}^N\hat w_i\pi_N(\hat c_i)^2 = \int\pi_N(y)^2\,d\rho_\nu(y)\ge c_1.\]
Since $c_1-\hat c_i\ge c_1-\hat c_N$ for every $i$, it follows that
\[
\begin{aligned}
(c_1-\hat c_N)c_1 &\le\sum_{i=1}^N\hat w_i(c_1-\hat c_i)\pi_N(\hat c_i)^2\\
&=\sum_{k\ge2}(c_1-c_k)c_k\pi_N(c_k)^2\\
&\le 4c_1S_1r^{-2(N-1)}.
\end{aligned}
\]
Dividing by $c_1$ proves the node estimate.

For the weight estimate, \cref{prop:nodes} supplies $\hat w_N\ge c_1$.
The node estimate supplies $N_\nu$ such that $\hat c_N\ge(c_1+c_2)/2$ for $N\ge N_\nu$.
For such $N$, by exactness and positivity,
\begin{equation}\label{eq:wexact}
\hat w_N\pi_N(\hat c_N)^2 \le \int\pi_N(y)^2\,d\rho_\nu(y) \le c_1+4S_1r^{-2(N-1)}.
\end{equation}
We will show that $\pi_N(\hat c_N)^{-2}\le1+C_\nu'Nr^{-2N}$ for all sufficiently large $N$, which combined with
\cref{eq:wexact} proves the weight estimate after enlarging $C_\nu$ and $N_\nu$ if necessary.
To do this, we put $\bar u=c_1/c_2>1$, so that $\varphi(\hat c_N)\ge\bar u$.
For $v>1$,
\[
T_{N-1}(v)=\cosh\bigl((N-1)\operatorname{arccosh}v\bigr),
\qquad
0\le(\log T_{N-1})'(v)\le\frac{N-1}{\sqrt{v^2-1}}.
\]
Consequently,
\begin{equation}\label{eq:logratio}
\begin{aligned}
0\le\log\frac{1}{\pi_N(\hat c_N)}&\le\frac{2(N-1)(c_1-\hat c_N)}{c_2\sqrt{\bar u^2-1}}\\
&\le\frac{8S_1(N-1)}{c_2\sqrt{\bar u^2-1}} r^{-2(N-1)}.
\end{aligned}
\end{equation}
The right-hand side is $O_\nu(Nr^{-2N})$ and tends to zero, exponentiation proves the claim, and the proof is complete.
\end{proof}

The estimates also control the exponential rate and Gamma shape appearing in
\cref{eq:gamma_tail}. Using $\hat c_N\ge a_1^2$ from \cref{prop:nodes}, we obtain
\begin{equation}\label{eq:tail_parameters}
\begin{aligned}
\delta_N:=\frac1{\hat c_N}-\frac1{c_1}&=\frac{c_1-\hat c_N}{c_1\hat c_N}=O_\nu(r^{-2N}),\\
\varepsilon_N:=b_N-1&=\frac{(\hat w_N-c_1)+(c_1-\hat c_N)}{\hat c_N} = O_\nu(Nr^{-2N}).
\end{aligned}
\end{equation}
We note that both quantities are positive and tend to zero; thus, in particular, the largest Gamma factor converges
in law to $\Gamma(1,c_1)$, the first exponential factor in \cref{eq:ct}.
The next result makes its implication for relative tail accuracy quantitative.

\begin{corollary}\label{cor:relative_tail}
Fix $\nu>-1$, and let $r$ be as in \cref{prop:extreme_rates}.
There are constants $C_\nu>0$, $\theta_\nu\in(0,1)$, and an integer $N_\nu$
such that, for $N\ge N_\nu$ and $1\le H\le r^{2N}$,
\begin{equation}\label{eq:relative_tail_window}
\sup_{0\le t\le H}
\left|\frac{1-F_N(t)}{1-F(t)}-1\right|
\le C_\nu\bigl(\theta_\nu^N+Hr^{-2N}\bigr).
\end{equation}
Consequently, for every positive sequence $t_N=o(r^{2N})$,
the relative survival error tends to zero uniformly on $[0,t_N]$.
\end{corollary}

\begin{proof}
Put $s_N=\hat c_N$, and use $\delta_N,\varepsilon_N$ from \cref{eq:tail_parameters}.
By the min--max principle for $J_N$, a compression of $J$ (\cref{rem:fredholm}),
we have that $\hat c_{N-1}\le c_2$ for $N\ge2$, so all remaining scales lie in $[0,c_2]$.

Set $\zeta=c_2/c_1<1$, and let
\[g(y)=-\frac{\log(1-y/c_1)}{y}\quad\text{and}\quad P_N(y)=\sum_{m=0}^{2N-1}\frac{y^m}{(m+1)c_1^{m+1}},\]
with $g(0)=1/c_1$. The power series gives $\sup_{[0,c_2]}|g-P_N|=O_\nu(\zeta^{2N}/N)$,
$P_N(c_1)=O_\nu(\log(N+1))$, and $\sup_{[0,c_1]}|P_N'|=O_\nu(N)$.
Exactness of the Gauss rule therefore yields
\[\log A_\nu-\sum_{i<N}\hat w_i g(\hat c_i) = \hat w_NP_N(s_N)-c_1P_N(c_1)+O_\nu(\zeta^{2N}/N).\]
The right side is $O_\nu(N\log(N+1)r^{-2N}+\zeta^{2N}/N)$ by \cref{eq:noderate}.
For large $N$, $s_N\ge s_0:=(c_1+c_2)/2>c_2$; the bounded parameter derivative
of $g$ then shows that replacing $c_1$ by $s_N$ changes the residual sum by $O_\nu(c_1-s_N)$.
Since $\log\Gamma(1+\varepsilon_N)+\varepsilon_N\log s_N=O_\nu(\varepsilon_N)$, we obtain
\begin{equation}\label{eq:tail_prefactor}
\log\frac{B_N}{A_\nu} = O_\nu\!\left(N\log(N+1)r^{-2N}+\frac{\zeta^{2N}}{N}\right).
\end{equation}

We next make the tail asymptotics uniform in $N$.
Choose $\tau\in(1/s_0,1/c_2)$. For $\widehat R_N=\sum_{i<N}G_i$, the bound on the remaining scales implies
\[\E[e^{\tau\widehat R_N}] = \prod_{i<N}(1-\tau\hat c_i)^{-b_i} \le \exp\!\left(\frac{\tau S_1}{1-\tau c_2}\right).\]
The same bound holds for $R=\sum_{k\ge2}c_ke_k$, and $0\le\varepsilon_N\le1$ for large $N$.
For $0\le\varepsilon\le1$ and $x>0$, the upper incomplete Gamma function satisfies
\[
\begin{gathered}
1\le\frac{\Gamma(1+\varepsilon,x)}{x^\varepsilon e^{-x}} = \int_0^\infty e^{-v}(1+v/x)^\varepsilon\,dv \le1+\frac{\varepsilon}{x},\\
\frac{\Gamma(1+\varepsilon,x)}{\Gamma(1+\varepsilon)} \le e^{-x}\left(1+\frac{x^\varepsilon}{\Gamma(1+\varepsilon)}\right).
\end{gathered}
\]
The upper bounds follow from concavity and from
$(x+v)^\varepsilon\le x^\varepsilon+v^\varepsilon$, respectively.
Split the conditional survival probability at $\widehat R_N=t/2$.
On $\{\widehat R_N\le t/2\}$, the first bound and
$(1-\widehat R_N/t)^{\varepsilon_N}=1+O(\varepsilon_N\widehat R_N/t)$
leave a relative error $O_\nu(\varepsilon_N/t)$ after exponential tilting.
The second bound makes the normalized contribution from $t/2<\widehat R_N\le t$
at most $C_\nu\E[e^{\widehat R_N/s_N}\bigm|\widehat R_N>t/2]
=O_\nu(e^{-\omega_\nu t})$, where $\omega_\nu=(\tau -1/s_0)/2>0$.
The exponential moment also controls $\mathbb P(\widehat R_N>t)$ after normalization.
Combining these estimates with \cref{rem:tails}, uniformly for large $N$ and $t\ge1$ we arrive at
\begin{equation}\label{eq:uniform_tail_asymptotics}
\begin{aligned}
\bar F_N(t) &= B_Nt^{\varepsilon_N}e^{-t/s_N}\left[1+O_\nu\!\left(\frac{\varepsilon_N}{t}+e^{-\omega_\nu t}\right)\right],\\
\bar F(t) &= A_\nu e^{-t/c_1}\left[1+O_\nu(e^{-\omega_\nu t})\right].
\end{aligned}
\end{equation}

Finally, fix $q=(1+q_*)/2$ and put $\ell=-(c_1/2)\log q>0$.
\Cref{thm:intro_convergence}(ii) gives
\[\sup_{0\le t\le\ell N}\left|\frac{\bar F_N(t)}{\bar F(t)}-1\right|\le C_\nu q^Ne^{\ell N/c_1}=C_\nu(\sqrt q)^N,\]
where we used the fact that $\bar F(t)\ge e^{-t/c_1}$, since $T\ge c_1e_1$.
This already proves \cref{eq:relative_tail_window} when $H<\ell N$, so we may assume that $H\ge\ell N$.
For $\ell N\le t\le H\le r^{2N}$, it follows from \cref{eq:uniform_tail_asymptotics} that
\[\frac{\bar F_N(t)}{\bar F(t)} = \frac{B_N}{A_\nu}\exp\!\left(\varepsilon_N\log t-\delta_Nt\right)
\left[1+O_\nu\!\left(\frac{\varepsilon_N}{t}+e^{-\omega_\nu t}\right)\right].\]
Here $\delta_Nt=O_\nu(Hr^{-2N})$ and $\varepsilon_N\log t=O_\nu(N^2r^{-2N})$ by \cref{eq:tail_parameters}.
Choose $\theta_\nu<1$ strictly larger than $\max\{\sqrt q,e^{-\omega_\nu\ell},r^{-2},\zeta^2\}$.
Then \cref{eq:tail_prefactor} and all the other terms except $\delta_Nt$ are $O_\nu(\theta_\nu^N)$.
The exponent is uniformly bounded as long as $H\le r^{2N}$, so exponentiation proves \cref{eq:relative_tail_window}.
The consequence follows by taking $H=\max\{1,t_N\}$ for all sufficiently large $N$.
\end{proof}

\end{document}